\documentclass{amsart}

\usepackage{amsmath}
\usepackage{amsthm}
\usepackage{amssymb}
\usepackage{amsfonts}
\usepackage{bm}
\usepackage[shortlabels]{enumitem}
\usepackage[bookmarks=true,hyperindex,pdftex,colorlinks,citecolor=red, linkcolor=blue]{hyperref}
\usepackage{centernot} 
\usepackage{marginnote} 
\usepackage{bbm}
\usepackage[all]{xy}
\usepackage{tikz}
\usepackage{comment}
\usetikzlibrary{arrows}
\usetikzlibrary{shapes,decorations}
\usetikzlibrary{positioning}
\usepackage{xcolor}

\definecolor{darkgreen}{rgb}{.2, .6, .2}

\newcommand{\ach}[1]{{\color{darkgreen}{#1}}} 

\DeclareMathOperator{\lspan}{span}                          
\DeclareMathOperator{\conv}{conv}                           
\DeclareMathOperator{\supp}{supp}                           
\DeclareMathOperator{\rad}{rad}                             
\DeclareMathOperator{\Lip}{Lip}                             
\DeclareMathOperator{\lip}{lip}                             

\newcommand{\N}{\mathbb{N}}             
\newcommand{\M}{\mathbb{M}}             
\newcommand{\R}{\mathbb{R}}             
\newcommand{\C}{\mathbb{C}}             
\newcommand{\K}{\mathbb{K}}             

\newcommand{\set}[1]{\left\{{#1}\right\}}                   
\newcommand{\norm}[1]{\left\|{#1}\right\|}                  
\newcommand{\dual}[1]{{#1}^\ast}                            

\newcommand{\lipfree}[1]{\mathcal{F}({#1})}                 
\newcommand{\F}{\mathcal{F}}                                
\newcommand{\FS}{\mathcal{FS}}                                
\newcommand{\lipnorm}[1]{\norm{#1}_L}                       

\newcommand{\restricted}{\mathord{\upharpoonright}}

\def\<{\langle}
\def\>{\rangle}
\newcommand{\ep}{\varepsilon}

\theoremstyle{plain}
\newtheorem{theorem}{Theorem}[section]
\newtheorem{lemma}[theorem]{Lemma}
\newtheorem{corollary}[theorem]{Corollary}
\newtheorem{proposition}[theorem]{Proposition}

\newtheorem*{ThmA}{Theorem~\ref{thmA}} 
\newtheorem*{ThmB}{Theorem~\ref{thmB}} 
\newtheorem*{ThmC}{Theorem~\ref{thmC}} 
\newtheorem{maintheorem}{Theorem} 

\theoremstyle{definition}
\newtheorem*{definition*}{Definition}

\newtheorem{example}[theorem]{Example}

\theoremstyle{remark}
\newtheorem{remark}[theorem]{Remark}

\begin{document}

\title[Compact and weakly compact Lipschitz operators]{Compact and weakly compact Lipschitz operators}

\author[A. Abbar]{Arafat Abbar}

\author[C. Coine]{Cl\'ement Coine}

\author[C. Petitjean]{Colin Petitjean}

\address[A. Abbar]{LAMA, Univ Gustave Eiffel, Univ Paris Est Creteil, CNRS, F--77447, Marne-la-Vall\'ee, France}
\email{arafat.abbar@univ-eiffel.fr}

\address[C. Coine]{Normandie Univ, UNICAEN, CNRS, LMNO, 14000 Caen, France}
\email{clement.coine@unicaen.fr}

\address[C. Petitjean]{LAMA, Univ Gustave Eiffel, Univ Paris Est Creteil, CNRS, F--77447, Marne-la-Vall\'ee, France}
\email{colin.petitjean@univ-eiffel.fr}

\date{}

\begin{abstract}
Any Lipschitz map $f : M \to N$ between two pointed metric spaces may be extended in a unique way to a bounded linear
operator $\widehat{f} : \F(M) \to \F(N)$ between their corresponding Lipschitz-free spaces. In this paper, we give a necessary and sufficient condition for $\widehat{f}$ to be compact in terms of metric conditions on $f$. This extends a result by A. Jim\'{e}nez-Vargas and M. Villegas-Vallecillos
in the case of non-separable and unbounded metric spaces. After studying the behavior of weakly convergent sequences made of finitely supported elements in Lipschitz-free spaces, we also deduce that $\widehat{f}$ is compact if and only if it is weakly compact.
\end{abstract}




\maketitle

\section{Introduction}
Let $(M,d)$ be a metric space equipped with a distinguished point denoted by $0_M \in M$. We let $\Lip_0(M)$ be the Banach space of Lipschitz maps from $M$ to $\K$ ($\K = \R$ or $\C$), vanishing at $0_M$, equipped with the norm
$$\displaystyle
\mathrm{Lip}(f) :=  \sup_{x \neq y \in M} \frac{|f(x)-f(y)|}{d(x,y)}.$$ 
For $x\in M$, we denote by $\delta(x)$  the bounded linear functional on $\Lip_0(M)$ defined by $\<f,\delta(x)\> = f(x), \ f\in \Lip_0(M).$ The Lipschitz-free space over $M$, denoted by $\F(M)$, is the Banach space
$$\F(M) := \overline{ \mbox{span}}^{\| \cdot  \|}\left \{ \delta(x) \, : \, x \in M  \right \} \subset \Lip_0(M)^*.$$
We refer the reader to \cite{GoKa_2003} or \cite{Weaver2} (where they are called Arens--Eells spaces) for more information on these spaces, including a proof of the next fundamental ``linearization" property which will be the cornerstone of our study.

\begin{proposition} \label{diagramfree}
Let $M$ and $N$ be two pointed metric spaces. Let $f \colon M \to N$ be a Lipschitz map such that $f(0_M) = 0_N$. Then, there exists a unique bounded linear operator $\widehat{f} \colon \F(M) \to \F(N)$ with $\|\widehat{f}\|=\mathrm{Lip}(f)$ and such that the following diagram commutes:
$$\xymatrix{
    M \ar[r]^f \ar[d]_{\delta_{M}}  & N \ar[d]^{\delta_{N}} \\
    \F(M) \ar[r]_{\widehat{f}} & \F(N)
  }.$$
  More precisely, for every $\gamma=\sum_{i=1}^n a_i\delta(x_i)\in \F(M)$,  
 $\widehat{f}(\gamma)=\sum_{i=1}^n a_i\delta(f(x_i))$.
\end{proposition}
In this paper, operators of the kind $\widehat{f} \colon \F(M) \to \F(N)$ will be called Lipschitz operators. 
The above linearization property carries some metric information about $f$ and the metric spaces $M,N$ themselves. Of course, passing from a Lipschitz map to a linear map has a price and the difficulty is to analyse the structure of the associated Lipschitz-free spaces. A very natural yet widely unexplored topic consists in the study of how metric properties of $f$ are transferred to linear properties of $\widehat{f}$, and vice-versa (see e.g. \cite{ACP20}). 
\smallskip

In this paper, we investigate the compactness properties of $\widehat{f}$ and characterize them in terms of metric conditions on $f$. Recall that an operator $T : X \to Y$ between Banach spaces is compact if the image by $T$ of the unit ball of $X$, denoted by $B_X$, is relatively compact in $Y$. Similarly, we say that $T$ is weakly compact if $T(B_X)$ is relatively weakly compact in $Y$. It is obvious that any compact operator is also weakly compact, while the converse is not true in general. 
A disguised study of compact Lipschitz operators has probably been initiated by Kamowitz and Scheinberg in \cite{Kamo} and then pursued by Jim\'{e}nez-Vargas and Villegas-Vallecillos in \cite{Vargas1} (see also \cite{JSV14} where vector-valued Lipschitz functions are considered). Indeed, in the last mentioned papers, the authors consider composition operators on Lipschitz spaces which appear naturally as the adjoints of our Lipschitz operators $\widehat{f}$. To be more specific, noting that
$$f\in \Lip_0(M) \mapsto \left[\sum_i a_i \delta(x_i) \mapsto \sum_i a_if(x_i) \right] \in \F(M)^*$$
is an isometric isomorphism, we get that $\big(\widehat{f}\,\big)^\ast=C_f$, where $C_f : \Lip_0(M) \to \Lip_0(N)$ is the composition operator given by $C_f(g) = g \circ f, \ g\in \Lip_0(M)$. Of course, by Schauder's theorem, $\widehat{f}$ is compact if and only if $\big(\widehat{f}\,\big)^*$ is compact, so one can tackle the problem either working with $C_f$ or working with $\widehat{f}$. In \cite{Vargas1}, the authors proved the next characterization. 
\smallskip

\noindent\textbf{Theorem (\cite[Theorem 1.2]{Vargas1}).}
\textit{
Let $M$ be pointed separable metric spaces and let $f: M \to M$ be a Lipschitz map vanishing at $0_M$. Assume that $M$ is bounded and separable. Then the composition operator $C_f : g \in \Lip_0(M) \mapsto g \circ f \in \Lip_0(M)$ is compact if and only if
\begin{enumerate}
	\item[(i)] $f(M)$ is totally bounded in $M$.  
	\item[(ii)] $f$ is uniformly locally flat, that is, for each
	$\ep > 0$, there exists $\delta > 0$ such that $d(f(x), f(y)) \leq \ep d(x,y)$ whenever $d ( x , y ) \leq  \delta$. 
\end{enumerate}
}
A few comments about the above statement are necessary. First, as it is proved in \cite[Theorem~8.7.8]{LipschitzBook}, the very same result holds for Lipschitz maps $f: M \to N$ where $N$ is any pointed metric space. Notice also that the separable assumption is absent in \cite[Theorem 1.2]{Vargas1}, but, as is this written in \cite{LipschitzBook}, the method of the proof needs $M$ to be separable. Finally the above condition $(ii)$ is called ``supercontractive'' in \cite{Vargas1}, but we also sometimes see it as the ``the little Lipschitz condition'' (since the space of uniformly locally flat Lipschitz functions is often called the little Lipschitz space, see \cite{Weaver2}).
\smallskip

Our first main result extends the previous theorem in the case of any metric spaces $M$ and $N$ (in particular not separable and unbounded). 
In fact, when $M$ is unbounded, one needs an additional assumption to take into account the behavior of the function $f$ at infinity. To prove our result, we are dealing directly with $\widehat{f}$ instead of its adjoint $C_f$.
Hence, even when $M$ is bounded, our proof is different from that of \cite{Vargas1}.

\begin{ThmA} 
Let $M,N$ be complete pointed metric spaces, and let $f : M \to N$ be a base point-preserving Lipschitz mapping. Then $\widehat{f} : \F(M) \to \F(N)$ is compact if and only if the next assertions are satisfied:
\begin{enumerate}
	\item[$(P_1)$]  For every bounded subset $S \subset M$, $f(S)$ is totally bounded in $N$;
	\item[$(P_2)$] $f$ is uniformly locally flat, that is, 
	$$ \lim\limits_{d(x,y) \to 0} \dfrac{d(f(x),f(y))}{d(x,y)} =0;$$
	\item[$(P_3)$] For every $(x_n,y_n)_n \subset \widetilde{M} : = \{(x,y) \in M \times M \; | \; x \neq y\}$ such that \\
	$\lim\limits_{n \to \infty} d(x_n,0) = \lim\limits_{n \to \infty} d(y_n,0) = \infty$, either 
	\smallskip
	
\begin{itemize}
	\item $(f(x_n), f(y_n))_n$ has an accumulation point in $N \times N$, or
	\item $\underset{n\to+\infty}{\liminf}\,\dfrac{d(f(x_n),f(y_n))}{d(x_n,y_n)}=0$.
\end{itemize}
\end{enumerate}
\end{ThmA}

It turns out that in the proof of ``$\implies$'' in Theorem~\ref{thmA}, which will be provided in Section~\ref{section2}, most of the time we only use the weaker assumption that $\widehat{f}$ is weakly compact. This suggests that there should be a close relationship between compact Lipschitz operators and weakly compact Lipschitz operators. Another clue is contained in \cite{JimenezWCompact}. Let us denote $\lip_0(M)$ the subspace of $\Lip_0(M)$ made of uniformly locally flat functions. Then we say that $\lip_0(M)$ separates the points (of $M$) uniformly if there exists $C >0$ such that, for every $x \neq y$, there exists a C-Lipschitz map $f \in \lip_0(M)$ with $|f(x) - f(y)| = d(x,y)$. Now 
\cite[Corllary~2.4]{JimenezWCompact} states that if $M$ is a compact metric space such that $\lip_0(M)$ separates the points uniformly, then the composition operator $C_f : g \in \Lip_0(M) \mapsto g \circ f \in \Lip_0(M)$ is weakly compact if and only if it is compact. Let us point out that for a compact metric space $M$, $\lip_0(M)$ separates points uniformly if and only if $M$ is purely 1-unrectifiable (that is, does not contain any bi-Lipschitz image of a subset of $\R$ with positive Lebesgue measure; see \cite[Theorem~A]{AGPP21}). This recent characterization underlines the fact that the assumptions in \cite[Corllary~2.4]{JimenezWCompact} are rather restrictive. We shall prove in Section~\ref{section3} that this result is actually true for every metric space $M$.

\begin{ThmB} 
	Let $M,N$ be complete pointed metric spaces, and let $f : M \to N$ be a base point-preserving Lipschitz mapping. The the next conditions are equivalent
	\begin{enumerate}
		\item $\widehat{f} : \F(M) \to \F(N)$ is compact;
		\item $\widehat{f} : \F(M) \to \F(N)$ is weakly compact;
		\item $C_f :  \Lip_0(N) \to \Lip_0(M)$ is compact;
		\item $C_f :  \Lip_0(N) \to \Lip_0(M)$ is weakly compact;
		\item $ C_f : \Lip_0(N) \to \Lip_0(M) \ \text{is weak}^*\text{-to-weak continuous}$.
	\end{enumerate}
\end{ThmB}

The key ingredient for proving Theorem~\ref{thmB} will be a structural result concerning weakly convergent sequences of finitely supported elements in Lipschitz-free spaces. We recall that $\gamma \in \F(M)$ is said to be finitely supported if $\gamma \in \lspan \left \{ \delta(x) \, : \, x \in M  \right \} $ and then the support of $\gamma$, denoted by $\supp (\gamma)$, is the smallest subset $S \subset M$ such that $\gamma \in \F(S)$. In what follows, for every $k \in \N$, $\mathcal{FS}_k(M)$ stands for the set of all $\gamma \in \F(M)$ such that $\supp(\gamma)$ contains at most $k$ points of $M$.

\begin{ThmC}
	Let $M$ be a complete metric space. If a sequence $(\gamma_n)_n \subset \mathcal{FS}_k(M)$ weakly converges to some $\gamma \in \F(M)$, then $\gamma \in \mathcal{FS}_k(M)$ and $(\gamma_n)_n$ actually converges to $\gamma$ in the norm topology. 
\end{ThmC}
The previous theorem can be deduced as a direct consequence of the deep result \cite[Theorem 5.2]{AlbiacKalton} and \cite[Lemma 2.10]{ACP20}. Since the proof from \cite{AlbiacKalton} is rather elaborated, for the convenience of the reader, we shall provide a different proof which is based on some recent developments in the theory of Lipschitz-free spaces.
\smallskip

\medskip

\noindent \textbf{Notation and background.} If $X$ is a Banach space, then we let $X^*$ be its topological dual, $B_X$ be its unit ball and $S_X$ be its unit sphere.  
\smallskip

Throughout the paper, $M,N$ are complete pointed metric spaces and the distinguished points will be denoted by $0_M$ and $0_N$ respectively, or simply $0$ if there is no ambiguity. We will write
    $$\widetilde{M} = \{(x,y) \in M \times M \; | \; x \neq y \}.$$
    We will use the notation
\begin{align*}
    B(p,r) &=  \{x \in M \; | \; d(x,p) \leq r \}\\
    \rad(S) &= \sup\{d(x,0) \; | \; x \in S\}
\end{align*}
where $p \in M$ and $S \subset M$. Next, if $(x_n)_n$ is a sequence of elements of $M$, we will say that \textit{$(x_n)_n$ goes to infinity} if $\lim_n d(x_n , 0_M) = \infty$. For convenience, let us recall the vector spaces
\begin{align*}
    \Lip(M) &= \{f \in \K^M \; | \; f \text{ is Lipschitz}\}\\
    \Lip_0(M) &= \{f \in \Lip(M) \; | \; f(0) = 0\} .
\end{align*}
We also wish to recall some important features of  the Lipschitz-free space over $M$,
$$\F(M) := \overline{ \mbox{span}}^{\| \cdot  \|}\left \{ \delta(x) \, : \, x \in M  \right \} \subset \Lip_0(M)^*.$$
First, $\F(M)$ is actually an isometric predual of $\Lip_0(M)$, that is $\F(M)^* \equiv \Lip_0(M)$. Moreover, if $0_M \in K \subset M$, then $\F(K)$ is isomorphic to a subspace of $\F(M)$ in the following way
$$\F(K) \simeq \overline{\text{span}} \{ \delta_M(x) \; | \; x \in K\} \subset \F(M).$$
According to this identification, the \textit{support of $\gamma \in \F(M)$} is the smallest closed subset $K \subset M$ such that $\gamma \in \F(K)$. It is denoted by $\text{supp}(\gamma)$. In particular and according to the terminology introduced before, $\mathcal{FS}_k(M)$ is the set of elements $\gamma \in \F(M)$ such that $\text{supp}(\gamma)$ is finite and $|\text{supp}(\gamma)| \leq k$ (where $|A|$ denote the cardinal of a subset $A \subset M$).
We refer to \cite{AP20, APPP2019} for more background information on the support. We mention here a very simple particular case of Theorem~\ref{thmC} in the case of some sequences in $\mathcal{FS}_1(M)$. We will use this fact in Section~$\ref{Section2}$ without mention, and it can be easily proved by considering the Lipschitz function $y \in M \mapsto d(x,y) - d(x,0_M)$.
\medskip

 \noindent \textbf{Fact:} \textit{If $(x_n)_n \subset M$ is such that $\delta(x_n) \to \delta(x)$ weakly, then $\delta(x_n) \to \delta(x)$ in the norm topology (which is equivalent to saying that $x_n \to x$ in $M$).}
\medskip

We also wish to mention that the Lipschitz-free space over $M$ is isometrically isomorphic to the Lipschitz-free space over its completion $\overline{M}$ in a very natural way. Indeed, it is readily checked that $f \in \Lip_0(\overline{M}) \mapsto f\restricted_M \in \Lip_0(M)$ is a weak$^*$-to-weak$^*$ continuous isometry. Hence, if $\overline{f} : \overline{M} \to \overline{N}$ is the unique extension of $f : M \to N$, then
$\widehat{f}: \F(M) \to \F(N)$ and $\widehat{\overline{f}} : \F(\overline{M}) \to \F(\overline{N})$ are conjugate one to another, so one of them is compact if and only if the other one is.
The only place where we need completeness is in Theorem \ref{thmA}. Indeed, in the proof, we use the fact that if $N$ is complete, then a totally bounded subset of $N$ is relatively compact. However, one could restate this theorem by replacing $N$ by its completion. 
So one can deduce the general statements (without completeness) from our statements (with completeness).
Since there is no real loss of generality, we will assume that $M$ and $N$ are always complete.
\smallskip

To conclude this introduction, let us state the next particular case of Urysohn's lemma that we shall use several times throughout the paper. It allows us to separate two or more points of $M$ by an element of $\text{Lip}_0(M)$. Since we are dealing with metric spaces, a concrete simple formula can be given for the Lipschitz map, but it can also be easily deduced from McShane extension's theorem, see e.g. \cite[Theorem 1.33 and Corollary 1.34]{Weaver2}.

\begin{lemma}
	\label{LemmaConstrctionLipF}
	Let $M$ be a pointed metric space, let $p\in M, p\neq 0_M$ and let $\ep \in (0, d(p,0_M)/4)$. Then there exists $f \in \text{Lip}_0(M)$ such that $f = 1$ on $B(p, \ep)$ and $f=0$ on $M \setminus B(p, 2 \ep).$
\end{lemma}

\section{A metric characterisation of compact Lipschitz operators}\label{Section2}
\label{section2}

The main objective of this section is to prove Theorem~\ref{thmA}. The proof will be based on the next easy but smart observation from \cite{Vargas2} (see Theorem~2.3 therein). This result concerns not only compact operators but also weakly compact operators, and so it will be useful in Section~$\ref{section3}$ as well. We shall provide its short proof for completeness.

\begin{proposition}[\cite{Vargas2}] \label{caracCompact}
	Let $M,N$ be pointed metric spaces and let $f : M \to N$ be a base point-preserving Lipschitz mapping. Then $\widehat{f} : \F(M) \to \F(N)$ is (weakly) compact if and only if 
	$$ \left\{ \dfrac{\delta(f(x)) - \delta(f(y))}{d(x,y)} \; | \; x \neq y \in M \right\} $$
	is relatively (weakly) compact in $\F(N)$.
\end{proposition}

\begin{proof} We will only prove the statement for compact operators, the proof being verbatim the same in the case of weakly compact operators. 
Notice that 
$$ \left\{ \dfrac{\delta(f(x)) - \delta(f(y))}{d(x,y)} \; | \; x \neq y \in M \right\} = \widehat{f}(\mathcal{M}),$$
where $\mathcal{M} = \left\{ d(x,y)^{-1}(\delta(x) - \delta(y)) \; | \; x \neq y \in M \right\}$. Since $\mathcal{M} \subset B_{\F(M)}$, if $\widehat{f}$ is  compact then $\widehat{f}(\mathcal{M})$ must be relatively compact. Conversely, it follows from the Hahn--Banach separation theorem that $B_{\F(M)} = \overline{\conv} \mathcal{M}$, the closure being taken for the norm topology. Now observe that 
$$ \widehat{f}(B_{\F(M)}) \subset  \widehat{f}(\overline{\conv}\mathcal{M}) \subset \overline{\conv} (\widehat{f}(\mathcal{M})) \subset \overline{\conv} \left(\overline{\widehat{f}(\mathcal{M})}\right).$$
So, if $\widehat{f}(\mathcal{M})$ is relatively compact, then $ \overline{\conv} \left(\overline{\widehat{f}(\mathcal{M})}\right)$ is compact (see e.g. \cite[Theorem~5.35]{IDBook}), and therefore  $\widehat{f}(B_{\F(M)})$ is relatively compact. 
\end{proof}

In the proof of Theorem~$\ref{thmA}$, we will use Proposition~$\ref{caracCompact}$ repeatedly and hence, we will work with sequences of finitely supported elements in Lipschitz-free spaces. By \cite[Lemma 2.10]{ACP20}, the set $\mathcal{FS}_k(M)$ of elements of $\F(M)$ whose support contains at most $k $ elements is weakly closed in $\F(M)$ (in particular, it is norm closed). We will use this fact in various places.

\begin{lemma}\label{Prepmaintheorem}
	Let $k\in \mathbb{N}$ and $(\gamma_n)_n = \big(\sum_{i=1}^k a_{i}(n) \delta(x_{i}(n))\big)_n \subset \mathcal{FS}_k(M)$ be a sequence converging weakly to an element $\gamma \in \mathcal{FS}_k(M)$. Then, for every $p\in \supp(\gamma)$, there exists $1\leq m \leq k$ such that $\underset{n\to+\infty}{\liminf}\,d(x_m(n),p)=0.$ 
\end{lemma} 

\begin{proof}
Let us write $\gamma = \sum_{i=1}^l a_i \delta(p_i)$ where $1\leq l \leq k$, $a_i \neq 0$ and $p_1, \ldots, p_l $ are pairwise distinct elements of $M \setminus \{0_M\}$.
Aiming for a contradiction, assume that there exists $1\leq j \leq l$ such that none of the sequences $(x_{i}(n))_n$, $1\leq i \leq k$, has a subsequence converging to $p_j$. Then, there exists $\ep > 0$ and a strictly increasing sequence $(n_m)_m \subset \mathbb{N}$ such that, for every $m$ and every $1\leq i\leq k$, $d(x_{i}(n_m), p_j) \geq \ep$. Hence, by Lemma~\ref{LemmaConstrctionLipF} we can find $h \in \Lip_0(M)$ such that $h(p_j) = 1, h(p_i) = 0$ if $i\neq j$ and $h = 0$ outside of $B(p_j, \ep/2)$.
Now, simply notice that since $\gamma_{n_m} \to \gamma$ in the weak topology we have
$$
0 = \< h , \gamma_{n_m} \> \to \< h , \gamma \> = a_j,
$$
which is a contradiction.
\end{proof}

\begin{lemma}\label{Prep2maintheorem} Let $f : M \to N$ be a Lipschitz map such that $f(0_M)=0_N$. Let $(x_n , y_n)_n \subset \widetilde{M}$ and let $(m_n)_n \subset \F(N)$ be defined by $$m_n=\dfrac{\delta(f(x_n)) - \delta(f(y_n))}{d(x_n,y_n)}.$$
Assume that $(m_n)_n$ weakly converges to $\gamma \in \F(N)$.
\begin{enumerate}
\item If $d(x_n,y_n) \to 0$ then $\gamma = 0$.
\item If $d(x_n,y_n) \to + \infty$ then $\gamma = 0$.
\item If there exists $\alpha>0$ such that $d(x_n,y_n) \geq \alpha$ and $\gamma \neq 0$ then $(d(x_n,y_n))_n$ is bounded and $(f(x_n),f(y_n))_n$ has an accumulation point in $N\times N$. 
\end{enumerate} 
\end{lemma} 

\begin{proof}
Notice that $(m_n)_n \subset \mathcal{FS}_2(N)$ which is weakly closed so $\gamma = a\delta(p)+b\delta(q) \in \mathcal{FS}_2(N)$ where either $p\neq q$ or $p=q=0$.
\smallskip

Let us prove $(1)$. If $\gamma \neq 0$ then we can assume that $a\neq 0$, $p \neq 0_N$ and, according to Lemma $\ref{Prepmaintheorem}$,  that $(f(x_n))_n$ or $(f(y_n))_n$ has a subsequence converging to $p$. Since $d(x_n,y_n) \to 0$, both subsequences converge, that is, there exists an increasing sequence $(n_k)_k \subset \N$ such that $(f(x_{n_k}))_k$ and $(f(y_{n_k}))_k$ are converging to $p$. The same lemma ensures that $b=0$ so that $\gamma_n \to a\delta(p)$.
Now, let $h \in \Lip_0(N)$ be such that $h$ takes the value 1 on a ball around $p$. Then, for $k$ large enough, $\langle h , m_{n_k} \rangle  = 0 $ while the limit over $k$ of this term is $\langle h , a \delta(p) \rangle = a$, therefore $a=0$. This is a contradiction so we must have $\gamma = 0$.
\smallskip

We now prove $(2)$. We aim for a contradiction. If $\gamma \neq 0$, we may assume that $a\neq 0$ and $p\neq 0_N$ and by Lemma~\ref{Prepmaintheorem}, up to extracting a subsequence, that $(f(x_n))_n$ converges to $p$. Since $(f(x_n))_n$ converges and $d(x_n,y_n) \to + \infty$, we have 
$$\|\cdot\| - \lim\limits_{n \to \infty} \dfrac{\delta(f(x_n))}{d(x_n,y_n)} = 0.$$ 
Therefore $( d(x_n,y_n)^{-1}\delta( f(y_n) )  )_n \subset \mathcal{FS}_1(N)$ must converge to an element $\gamma'=c\delta(r)$. We then distinguish two cases :
\begin{itemize}
	\item[$\bullet$] If for some subsequence, $(f(y_{n_k}))_k$ is bounded, then $m_{n_k} \to 0$ and this is a contradiction.
	\item[$\bullet$] If for some subsequence, $d(f(y_{n_k}),0) \to +\infty$, then $(f(y_{n_k}))_k$ is eventually far from $r$. Similarly as in $(1)$, one can show that $c$ must be equal to 0 by using a Lipschitz map taking the value 1 at $p$ and 0 outside of a ball centered at $r$. So $m_{n_k} \to 0$, yet another contradiction.
\end{itemize}
\smallskip

Let us finish with the proof of $(3)$. As above, since $\gamma \neq 0$ we can assume that $a\neq 0$, $p\neq 0_N$ and $(f(x_n))_n$ converges to $p$. We only need to show that $(f(y_n))_n$ has a convergent subsequence. If $b\neq 0$ and $q$ is not equal to $0_N$ or $p$, then Lemma~\ref{Prepmaintheorem} ensures that $(f(y_n))_n$ has a subsequence converging to $q$. So assume that $b=0$ or $q=0_N$, that is, $\gamma_n \to a\delta(p)$.
Up to extracting another subsequence, we may assume that $d(x_n,y_n)$ converges to $\rho \in (0 , +\infty]$. If $\rho = + \infty$ then $\gamma = 0$ by $(2)$, so we actually have that $\rho \in (0 , +\infty)$. Therefore
$$
\delta(f(x_n)) - \delta(f(y_n)) \to a' \delta(p) \ \ \text{weakly}
$$
where $a'=a\rho$. Since $\delta(f(x_n)) \to \delta(p)$, we have
$$\delta(f(y_n)) \to a'' \delta(p) \ \ \text{weakly}$$
where $a''=1-a'$. If $a'' \neq 0$ then by Lemma $\ref{Prepmaintheorem}$, $f(y_n)$ has a subsequence converging to $p$, and if $a''=0$ then $f(y_n) \to 0_N$. 
\end{proof}

We need one last lemma before the proof of Theorem~\ref{thmA}. For convenience, let us recall we was called $(P_3)$ in this statement.
\begin{enumerate}
	\item[$(P_3)$] For every $(x_n,y_n)_n \subset \widetilde{M} : = \{(x,y) \in M \times M \; | \; x \neq y\}$ such that $\lim\limits_{n \to \infty} d(x_n,0) = \lim\limits_{n \to \infty} d(y_n,0) = \infty$, either 
	\smallskip
	
	\begin{itemize}
		\item $(f(x_n), f(y_n))_n$ has an accumulation point in $N \times N$, or
		\item $\underset{n\to+\infty}{\liminf}\,\dfrac{d(f(x_n),f(y_n))}{d(x_n,y_n)}=0$.
	\end{itemize}
\end{enumerate}

\begin{lemma}\label{3implies5} 
Let $M$ be an unbounded metric space, $N$ be any metric space and $f : M \to N$ be any map. If $f$ satisfies $(P_3)$ then $f$ is radially flat, that is
$$\ \  \lim\limits_{d(x,0) \to \infty} \dfrac{d(f(x),0)}{d(x,0)} =0. $$
\end{lemma}

\begin{proof}
Assume that $f$ satisfies $(P_3)$. Let $(x_n)_n \subset M$ be such that $d(x_n,0) \to +\infty$. We will show that there exists a subsequence $(x_{n_k})_k$ such that $$\dfrac{d(f(x_{n_k}),0)}{d(x_{n_k},0)} \underset{k\to +\infty}{\longrightarrow} 0.$$
In view of applying Property $(P_3)$, we first construct  by induction an increasing sequence $(n_k)_k \subset \mathbb{N}$ and a sequence $(y_{n_k})_k \subset M$ such that for every $k \in \N$
\begin{itemize}
    \item[(i)] $d(y_{n_k},0) \geq k$;
    \item[(ii)] $d(x_{n_k},y_{n_k}) \geq k$;
    \item[(iii)] $\dfrac{d(x_{n_k}, y_{n_k})}{d(x_{n_k}, y_{n_k}) - d(y_{n_k}, 0)} \leq 2$;
    \item[(iv)] $\dfrac{d(f(y_{n_k}), 0)}{d(x_{n_k}, y_{n_k}) - d(y_{n_k}, 0)} \leq \dfrac{1}{k}$.
\end{itemize}

We proceed by induction and start with the base case $k=1$. Since $M$ is unbounded, there is an element $y\in M$ such that $d(y,0) \geq 1$. We fix such $y$. The inequality
$$
d(x_n,y) \geq d(x_n,0)-d(y,0)
$$
yields $d(x_n,y) \underset{n\to +\infty}{\longrightarrow} +\infty$ so that
$$
\dfrac{d(x_n, y)}{d(x_n, y) - d(y, 0)} \underset{n\to +\infty}{\longrightarrow} 1 \ \ \ \text{and} \ \ \ \dfrac{d(f(y), 0)}{d(x_n, y) - d(y, 0)} \underset{n\to +\infty}{\longrightarrow} 0.
$$
Hence, we can find $n_1 \in \mathbb{N}$ large enough so that
$$
d(x_{n_1},y) \geq 1, \ \dfrac{d(x_{n_1}, y)}{d(x_{n_1}, y) - d(y, 0)} \leq 2 \ \ \ \text{and} \ \ \ \dfrac{d(f(y), 0)}{d(x_{n_1}, y) - d(y, 0)} \leq 1.
$$
We then set $y_{n_1}=y$.

Assume now that  $y_{n_1}, \ldots, y_{n_k} \in M$ are constructed with $n_1 < n_2 < \cdots < n_k$. We can find $y\in M$ such that $d(y,0) \geq k+1$. We now proceed as above, and we find $n_{k+1} \in \mathbb{N}$ such that $n_k < n_{k+1}$ and
$$
d(x_{n_{k+1}},y) \geq k+1, \ \dfrac{d(x_{n_{k+1}}, y)}{d(x_{n_{k+1}}, y) - d(y, 0)} \leq 2 \ \ \ \text{and} \ \ \ \dfrac{d(f(y), 0)}{d(x_{n_{k+1}}, y) - d(y, 0)} \leq \dfrac{1}{k+1}.
$$
We can now set $y_{n_{k+1}} = y$ and by construction, the sequence $(y_{n_k})_k \subset M$ satisfies the desired properties.\\
In particular, $d(y_{n_k},0) \to +\infty$, so we can apply $(P_3)$ to $(x_{n_k})_k$ and $(y_{n_k})_k$ and we keep denoting by $(x_{n_k})_k$ and $(y_{n_k})_k$ the subsequences that we obtain. Hence, we  either have $f(x_{n_k}) \to p$ and $f(y_{n_k}) \to q$ for some $p,q \in N$ or $\frac{d(f(x_{n_k}) , f(y_{n_k}))}{d(x_{n_k},y_{n_k})} \to 0$. Note that if we are in the first case, then we also have 
\begin{equation}\label{CV0}
    \dfrac{d(f(x_{n_k}) , f(y_{n_k}))}{d(x_{n_k},y_{n_k})} \to 0
\end{equation}
because $d(x_{n_k},y_{n_k}) \to +\infty$,
and that is the property we will need. Indeed, by the triangle inequality
\begin{align*}
\dfrac{d(f(x_{n_k}), 0)}{d(x_{n_k},0)} & \leq \dfrac{d(f(x_{n_k}), f(y_{n_k}))}{d(x_{n_k},y_{n_k}) - d(y_{n_k},0)} + \dfrac{d(f(y_{n_k}), 0)}{d(x_{n_k},y_{n_k}) - d(y_{n_k},0)} \\
& = \dfrac{d(f(x_{n_k}), f(y_{n_k}))}{d(x_{n_k}, y_{n_k})} \dfrac{d(x_{n_k},y_{n_k})}{d(x_{n_k},y_{n_k}) - d(y_{n_k},0)} + \dfrac{d(f(y_{n_k}), 0)}{d(x_{n_k},y_{n_k}) - d(y_{n_k},0)}
\end{align*}
and the right hand side converges to $0$ by $(iii)$, $(iv)$ and \eqref{CV0}.
\end{proof}

\begin{maintheorem} \label{thmA}
	Let $M,N$ be complete pointed metric spaces, and let $f : M \to N$ be a base point-preserving Lipschitz mapping. Then $\widehat{f} : \F(M) \to \F(N)$ is compact if and only if the next assertions are satisfied:
	\begin{enumerate}
		\item[$(P_1)$]  For every bounded subset $S \subset M$, $f(S)$ is totally bounded in $N$;
		\item[$(P_2)$] $f$ is uniformly locally flat, that is, 
		$$ \lim\limits_{d(x,y) \to 0} \dfrac{d(f(x),f(y))}{d(x,y)} =0;$$
		\item[$(P_3)$] For every $(x_n,y_n)_n \subset \widetilde{M} : = \{(x,y) \in M \times M \; | \; x \neq y\}$ such that\\
		$\lim\limits_{n \to \infty} d(x_n,0) = \lim\limits_{n \to \infty} d(y_n,0) = \infty$, either 
		\smallskip
		
		\begin{itemize}
			\item $(f(x_n), f(y_n))_n$ has an accumulation point in $N \times N$, or
			\item $\underset{n\to+\infty}{\liminf}\,\dfrac{d(f(x_n),f(y_n))}{d(x_n,y_n)}=0$.
		\end{itemize}
	\end{enumerate}
\end{maintheorem}

\begin{remark}
Assume that the condition $(P_3)$ is satisfied. Then, if $(x_n,y_n)_n \subset \widetilde{M}$ is such that $\dfrac{d(f(x_n),f(y_n))}{d(x_n,y_n)}$ does not converge to $0$, there is a subsequence $(x_{n_k},y_{n_k})_k$ such that $\underset{k\to+\infty}{\liminf}\,\dfrac{d(f(x_{n_k}),f(y_{n_k}))}{d(x_{n_k},y_{n_k})}>0$. This implies that $(f(x_{n_k}), f(y_{n_k}))_k$ and hence $(f(x_n), f(y_n))_n$, has an accumulation point in $N \times N$. This tells us that we can reformulate condition $(P_3)$ by :

\begin{enumerate}
    \item[$(P_3')$] For every $(x_n,y_n)_n \subset \widetilde{M} : = \{(x,y) \in M \times M \; | \; x \neq y\}$ such that
		$\lim\limits_{n \to \infty} d(x_n,0) = \lim\limits_{n \to \infty} d(y_n,0) = \infty$, either 
		\smallskip
		
		\begin{itemize}
			\item $(f(x_n), f(y_n))_n$ has an accumulation point in $N \times N$, or
			\item $\underset{n\to+\infty}{\lim}\,\dfrac{d(f(x_n),f(y_n))}{d(x_n,y_n)}=0$.
		\end{itemize}
		\end{enumerate}
\end{remark}

\smallskip

\begin{proof}[Proof of Theorem $\ref{thmA}$.]
We first prove the ``$\implies$" direction. 
\medskip

We start with $\widehat{f}$ compact implies $(P_1)$. Let $S$ be a bounded subset of $M$ and let $(x_n)_n$ be a sequence in $S$. By assumption (and Proposition \ref{caracCompact}), the sequence
$$ (m_n)_n := \left(\widehat{f}\left(\frac{\delta(x_n)}{d(x_n,0)} \right)\right)_n = \left(\frac{\delta(f(x_n))}{d(x_n,0)}\right)_n$$
has a convergent subsequence $(m_{n_k})_k$.
Denote by $\gamma$ the limit of $(m_{n_k})_k$. If $\gamma = 0$, then
$$
d(f(x_{n_k}), 0_N) =\| m_{n_k} \| d(x_{n_k}, 0_M) \underset{k \to \infty}{\longrightarrow} 0
$$
because $(x_{n_k})_k$ is bounded. In that case, $f(x_{n_k}) \to 0_N$ and we are done. Hence, it only remains to consider the case when $\gamma \neq 0$. By Lemma $\ref{Prep2maintheorem}$, this can only happen if $d(x_{n_k},0_M)$ does not tend to $0$. But then, we can find a subsequence, still denoted by $(n_k)_k$ for convenience, such that $d(x_{n_k},0_M) \geq \alpha > 0$ for every $k$. By the same Lemma, we then must have a subsequence of $(f(x_{n_k}))_k$ which converges and this finishes to prove $(P_1)$.

\medskip

We now show that $\widehat{f}$ compact implies $(P_2)$. Let $(x_n)_n$, $(y_n)_n$ be two sequences in $M$ such that $d(x_n, y_n) \to 0$. By Proposition \ref{caracCompact}, the sequence
$$ \left(\frac{\delta(f(x_n)) - \delta(f(y_n))}{d(x_n,y_n)}\right)_n$$
has a converging subsequence. However it follows immediately from Lemma~\ref{Prep2maintheorem}~(1) that the limit is $0$.

\medskip

It remains to prove that $\widehat{f}$ compact implies $(P_3)$. We already know that if $\widehat{f}$ is compact then $f$ satisfies $(P_2)$, which will be of use. Let $(x_n)_n, (y_n)_n \subset M$ going to infinity with $x_n \neq y_n$. Again, we let
$$m_n := \dfrac{\delta(f(x_n)) - \delta(f(y_n))}{d(x_n,y_n)}$$
and $(m_n)_n$ has a convergent subsequence, which we keep denoting by $(m_n)_n$, for simplicity. Let $\gamma$ be the limit of $(m_n)_n$. Notice that 
$$ \|m_n \| = \dfrac{d(f(x_n) , f(y_n))}{d(x_n,y_n)}.$$

We distinguish two cases: up to extracting a further subsequence, we will need to consider the cases when $d(x_n,y_n)$ converges to $0$ and when there exists $\alpha >0$ such that $d(x_n,y_n) \geq \alpha$ for every $n$. In the first case, we get by $(P_2)$ that $m_n \to 0$ so that $\|m_n\| \to 0$. In the second case, if $\gamma \neq 0$, we have by Lemma~\ref{Prep2maintheorem}~(3) that there exist $p,q\in N$ and an increasing sequence $(n_k)_k \subset \N$ such that $f(x_{n_k}) \to p$ and $f(y_{n_k}) \to q$. Finally if $\gamma = 0$ then again $\|m_n\| \to 0$. In all cases, $f$ satisfies $(P_3)$.
\bigskip

Let us now prove the ``$\impliedby$" direction.  We keep using the notation
$$(m_n)_n :=\left(\frac{\delta(f(x_n)) - \delta(f(y_n))}{d(x_n,y_n)}\right)_n$$
where $x_n \neq y_n \in M$ for every $n \in \N$. By Proposition $\ref{caracCompact}$, we have to show that this sequence admits a convergent subsequence in $\F(N)$. Up to extracting a subsequence, we only have to distinguish three cases : when both $(x_n)_n$ and $(y_n)_n$ are bounded, when one of them is bounded while the other one goes to $+\infty$, and when both go to $+\infty$.

\begin{enumerate}[(i), leftmargin=*,itemsep=5pt]
    \item If $(x_n)_n$ and $(y_n)_n$ are bounded, by $(P_1)$ there exists an increasing sequence $(n_k)_k \subset \N$ such that $(f(x_{n_k}))_k$ converges to a point $p \in N$ and $(f(y_{n_k}))_k$ converges to some $q\in N$. Since the sequence $(d(x_{n_k},y_{n_k}))_k$ is bounded, up to a further extraction, we may assume that it converges to some $\rho \geq 0$. Since $f$ is uniformly locally flat, if $\rho = 0$ then $(m_{n_k})_k$ converges to 0. If $\rho>0$, then it is readily seen that $(m_{n_k})_k$ converges to $\rho^{-1}(\delta(p) - \delta(q))$.
    
    \item If $(x_n)_n$ is bounded while $d(y_n,0) \to \infty$, thanks to $(P_1)$ there exists an increasing sequence $(n_k)_k \subset \N$ such that
    $(f(x_{n_k}))_k$ converges to a point $p \in N$. Therefore we may write for every $k \in \N$:
    \begin{align*}
        m_{n_k} &= \frac{\delta(f(x_{n_k})) - \delta(f(y_{n_k}))}{d(x_{n_k},y_{n_k})}\\
        &= \frac{\delta(f(x_{n_k})) - \delta(0)}{d(x_{n_k},y_{n_k})} + \frac{\delta(0) - \delta(f(y_{n_k}))}{d(0,y_{n_k})}  \frac{d(0,y_{n_k})}{d(x_{n_k},y_{n_k})}. 
    \end{align*}
    On the one hand,
    $$ \left\|\frac{\delta(f(x_{n_k})) - \delta(0)}{d(x_{n_k},y_{n_k})}\right\| = \frac{d(f(x_{n_k}) , 0)}{d(x_{n_k},y_{n_k})}  \underset{k \to \infty}{\longrightarrow} 0.$$
    On the other hand, $f$ is radially flat thanks to Lemma~$\ref{3implies5}$ so that
    $$  \left\|\frac{\delta(0) - \delta(f(y_{n_k}))}{d(0,y_{n_k})}\right\| = \frac{d(f(y_{n_k}) , 0)}{d(y_{n_k},0)}  \underset{k \to \infty}{\longrightarrow} 0.  $$
    Since the triangle inequality implies that 
    $\lim\limits_{k \to \infty} d(0 , y_{n_k})^{-1} d(x_{n_k},y_{n_k}) = 1$, we obtain that $(m_{n_k})_k$ converges to 0.

    \item If $d(x_n,0) \to +\infty$ and $d(y_n,0)\to +\infty$, then by $(P_3)$ there exists $(n_k)_k \subset \N$ such that, either $\|m_{n_k}\| \to 0$ or $f(x_{n_k}) \to p$ and $f(y_{n_k}) \to q$ for some $p,q \in N$. In the first case we are done since $(m_{n_k})_k$ converges to 0. In the second case, up to further extraction, we may assume that $d(x_{n_k},y_{n_k}) \to \rho \in [0,+\infty]$. Hence, $m_{n_k}$ converges to $0$ if $\rho = 0$ or $\rho = +\infty$ and converges to $\rho^{-1}(\delta(p) - \delta(q))$ otherwise.
\end{enumerate}
In all cases, the sequence $(m_n)_n$ admits a convergent subsequence.
\end{proof}

Of course, condition $(P_3)$ is always satisfied if the metric space $M$ is bounded. Similarly, condition $(P_2)$ is always satisfied if the space is uniformly discrete, that is, $\inf_{x\neq y} d(x,y) >0$. On the other hand, if $M=\R=N$ with the usual metric $|.|$, this condition means that $f'=0$ and hence $f=0$ because $f(0)=0$. In particular, according to Theorem~\ref{thmA}, the only compact Lipschitz operator $\widehat{f} : \F(\R) \to \F(\R)$ is $0$.
Furthermore, $(P_3)$ may seem uneasy to check. The next result shows that we may replace this property by a stronger yet simpler condition. Nonetheless, Example~$\ref{P4notnecessary}$ will show that this condition is not necessary.

\begin{corollary}
Let $M,N$ be complete pointed metric spaces, and let $f : M \to N$ be a base point-preserving Lipschitz mapping. If $f$ satisfies
\begin{enumerate}
    \item[$(P_1)$] For every bounded subset $S \subset M$, $f(S)$ is totally bounded in $N$;
    \item[$(P_2)$] $f$ is uniformly locally flat, that is, 
    $$ \lim\limits_{d(x,y) \to 0} \dfrac{d(f(x),f(y))}{d(x,y)} =0;$$
    \item[$(P_4)$] $f$ is flat at infinity, that is,  
        $$ \underset{d(y,0) \to \infty}{\lim\limits_{d(x,0) \to \infty}}  \dfrac{d(f(x),f(y))}{d(x,y)} =0,$$
\end{enumerate}
then $\widehat{f} : \F(M) \to \F(N)$ is compact.
\end{corollary}

\begin{proof}
It is readily seen that $(P_4)$ implies $(P_3)$.
\end{proof}

\begin{remark}
Assume that $\widehat{f}$ is compact. It follows from Proposition $\ref{caracCompact}$ (or the proof of Theorem $\ref{thmA}$) and Lemma $\ref{Prep2maintheorem}$ that $f$ satisfies the following property
$$
\underset{d(x,y)\to +\infty}{\lim} \ \dfrac{d(f(x),f(y))}{d(x,y)} = 0.
$$
This property is stronger than the condition ``radially flat'' from Lemma $\ref{3implies5}$, but weaker than the condition ``flat at infinity'' from the previous corollary.
\end{remark}

\begin{example}[Property $(P_4)$ is not necessary]\label{P4notnecessary}
Consider $(M,d) =(\N \cup \{0\}, |\cdot|)$ and $f : M \to M$ obtained by $f(2n) = 0$ and $f(2n+1) = 1$. Then $f$ is clearly Lipschitz and $\widehat{f} : \F(M) \to \F(M)$ is compact because its range is finite dimensional. Even so, if we let $x_n = 2n+1$ and $y_n=2n$ then $d(x_n,0), d(y_n,0) \to +\infty$ while $\frac{d(f(x_n), f(y_n))}{d(x_n,y_n)} = 1$ for every $n$. Consequently $f$ does not satisfy $(P_4)$.
\end{example}

In fact, in the previous example, $f$ satisfies a much stronger property: $f(M)$ is totally bounded. 

\begin{corollary}
Let $M,N$ be complete pointed metric spaces, and let $f : M \to N$ be a base point-preserving Lipschitz mapping. If $f(M)$ is totally bounded in $N$ and $f$ is uniformly locally flat, then $\widehat{f} : \F(M) \to \F(N)$ is compact.
\end{corollary}

\begin{proof}
If $f(M)$ is totally bounded then clearly $f$ satisfies $(P_1)$. Moreover if $(x_n)_n$ and $(y_n)_n$ are two sequences in $M$ going to infinity, then the sequences $(f(x_n))_n$ and $(f(y_n))_n$ have a common convergent subsequence and so $f$ readily satisfies $(P_3)$ in Theorem~$\ref{thmA}$.
\end{proof}

\begin{example}[$f(M)$ totally bounded is not necessary]\label{P4notnecessary2}
Take $M= \N \cup \{0\}$ equipped with the metric given by $d(n,0)=n!$ and $d(n,m)=n!+m!$ if $n\neq m$. Define $f : M \to M$ by $f(0)=0$ and $f(n) = n-1$ if $n\geq 1$. Then $f(M)=M$ is clearly not totally bounded while $\widehat{f}$ is compact as it satisfies $(P_1), (P_2)$ and $(P_4)$.
\end{example}

\section{Weak compactness of Lipschitz operators}\label{section3}

As we already mentioned in the introduction, Theorem~\ref{thmB} is an easy consequence of Theorem~\ref{thmC}, which states that norm-convergence and weak-convergence are equivalent for sequences in $\mathcal{FS}_k(M)$, plus some other classical results concerning (weakly) compact operators. We postpone the proof of Theorem~\ref{thmC} in order to first discuss its use in the proof of Theorem~\ref{thmB}.

\begin{maintheorem} \label{thmB}	Let $M,N$ be complete pointed metric spaces, and let $f : M \to N$ be a base point-preserving Lipschitz mapping. The the next conditions are equivalent
	\begin{enumerate}
		\item $\widehat{f} : \F(M) \to \F(N)$ is compact;
		\item $\widehat{f} : \F(M) \to \F(N)$ is weakly compact;
		\item $C_f :  \Lip_0(N) \to \Lip_0(M)$ is compact;
		\item $C_f :  \Lip_0(N) \to \Lip_0(M)$ is weakly compact;
		\item $ C_f : \Lip_0(N) \to \Lip_0(M) \ \text{is weak}^*\text{-to-weak continuous}$.
	\end{enumerate}
\end{maintheorem}

\begin{proof}
The implication $(1) \implies (2)$ is obvious.  Next, $(2) \implies (1)$ follows from Theorem~\ref{thmC} and Proposition~\ref{caracCompact}. Indeed, thanks to the Eberlein--\v{S}mulian theorem (see \cite[Theorem~1.6.3]{TopicsBanachSpaceTheory} e.g.), a subset $S$ of a Banach space $X$ is (relatively) weakly compact if and only if it is (relatively) weakly sequentially compact. So, Theorem~\ref{thmC} implies that a subset $S \subset \mathcal{FS}_k(M)$ is weakly compact if and only if it is compact in the norm topology. Now observe that the set appearing in Proposition~\ref{caracCompact} is a subset of $\mathcal{FS}_2(M)$ so that compactness and weak compactness are indeed equivalent.  To conclude, $(1) \iff (3)$ follows from Schauder's theorem (see e.g. \cite[Theorem~3.4.15]{Megg}), $(2) \iff (4)$ follows from Gantmacher's theorem (see e.g. \cite[Theorem~3.5.13]{Megg}), and $(2) \iff (5)$ follows from a classical result \cite[Theorem~3.5.14]{Megg} due to Gantmacher in the separable case and Nakamura in the general case.
\end{proof}

Theorem~\ref{thmC} is essentially contained in the very deep result \cite[Theorem~5.2]{AlbiacKalton}, even if one really needs to use the weak closeness of $\mathcal{FS}_k(M)$  \cite[Lemma 2.10]{ACP20} in order to obtain the statement we give.
For the sake of completeness, we will take advantage of some recent developments in the study of Lipschitz-free spaces in order to provide a new direct proof of this result. First, we recall two useful facts.
The first one shows that the pointwise multiplication with a Lipschitz function of bounded support always results in a Lipschitz function and, in fact, defines a continuous operator between Lipschitz spaces.

\begin{lemma}[Lemma~2.3 in \cite{APPP2019}]
\label{lm:multiplication_operator}
Let $M$ be a pointed metric space and let $h\in\Lip(M)$ have bounded support. Let $K\subset M$ contain the base point and the support of $h$. For $f\in\Lip_0(K)$, let $T_h(f)$ be the function given by
\begin{equation}
\label{eq:T_h}
T_h(f)(x)=\begin{cases}
f(x)h(x) & \text{if } x\in K \\
0 & \text{if } x\notin K
\end{cases} \,.
\end{equation}
Then $T_h$ defines a  weak$^*$-to-weak$^*$ continuous linear operator from $\Lip_0(K)$ into $\Lip_0(M)$, and $\norm{T_h}\leq\norm{h}_\infty+\rad(\supp(h))\lipnorm{h}$.
\end{lemma}

The function $T_h(f)$ does not depend on the choice of $K$, as long as it contains the support of $h$. Thus the requirement that $0\in K$ is not really a restriction, as one may always use the set $K\cup\set{0}$ instead.
Since $T_h$ is weak$^*$-to-weak$^*$ continuous, there is an associated bounded linear operator $W_h\colon\lipfree{M}\rightarrow\lipfree{K}$ such that $\dual{W_h}=T_h$.
\smallskip

The second fact is the following, whose proof can be deduced from that of \cite[Lemma~4.5]{Kalton04}.

\begin{lemma} \label{lemmaKalton} 
Let $M$ be a bounded metric space. If $(\gamma_n)_n \subset \F(M)$ is a weakly null sequence such that 
$$ \exists \ep >0, \forall n \neq m , \quad d(\supp(\gamma_n) , \supp(\gamma_m))> \ep,$$
then $(\gamma_n)_n$ converges to 0 in the norm topology.
\end{lemma}

We are now ready to prove the desired structural result about finitely supported sequences in Lipschitz-free spaces. 

\begin{maintheorem} \label{thmC} 
Let $M$ be a complete metric space. If a sequence $(\gamma_n)_n \subset \mathcal{FS}_k(M)$ weakly converges to some $\gamma \in \F(M)$, then $\gamma \in \mathcal{FS}_k(M)$ and $(\gamma_n)_n$ converges to $\gamma$ in the norm topology.
\end{maintheorem}

\begin{proof}
Since $\mathcal{FS}_k(M)$ is weakly closed by \cite[Lemma~2.10]{ACP20},  if a sequence $(\gamma_n)_n \subset \mathcal{FS}_k(M)$ weakly converges to some $\gamma \in \F(M)$, then $\gamma \in \mathcal{FS}_k(M)$. Therefore, for every $n \in \N$, $\gamma - \gamma_n \in \mathcal{FS}_{2k}(M)$. Consequently, to prove the result it is enough to show that for every complete metric space $M$ and for every $k \in \N$, any weakly null sequence in $\FS_k(M)$ is actually norm null. We will proceed by induction on $k \in \N$. Thanks to \cite[Theorem~A]{AACD20}, there exists a bounded metric space $B(M)$ such that $\F(M)$ is linearly isomorphic to $\F(B(M))$. Moreover the isomorphism $T : \F(M) \to \F(B(M))$ preserves finitely supported elements is the sense that $\gamma \in \FS_k(M)$ if and only if $T(\gamma) \in \FS_k(B(M))$. So, without loss of generality, we may assume that $M$ is a bounded metric space. 

If $k=1$ and $(\gamma_n)_n \subset \mathcal{FS}_1(M)$ is a weakly null sequence, we can write $\gamma_n = a_n \delta(x_n)$ where $a_n \in \K$ and $x_n \in M$. Let us denote $f := d(\cdot,0) \in \Lip_0(M)$. Since $(\gamma_n)_n$ is weakly null, it is readily seen that 
$$\|\gamma_n\| = |a_n| d(x_n,0)  = |\langle f , \gamma_n \rangle | \underset{n \to \infty}{\longrightarrow} 0.$$

Let us fix $k \in \N$. Assume we have shown that, for every $j \leq k$, every weakly null sequence in $\mathcal{FS}_j(M)$ is in fact norm null. Let us consider a weakly null sequence $(\gamma_n)_n \subset \mathcal{FS}_{k+1}(M)$. For every $n \in \N$, we will write 
$$\gamma_n = \sum_{i=1}^{k+1} a_i(n) \delta(x_i(n)),$$
where $a_i(n) \in \K$ and $x_i(n) \in M$ for every $1 \leq i \leq k+1$. We will distinguish two cases: 
\begin{itemize}[leftmargin=0pt, itemsep=4pt]
    \item There exists $i \in \{1, \ldots , k+1\}$ such that $(x_i(n))_n$ has a convergent subsequence to some $x \in M$. For simplicity, we still denote the subsequence by $(x_i(n))_n$. Notice that $i \in \{1, \ldots , k+1\}$ as above might not be unique. So, up to a further extraction, we may assume that there exists $\ep >0$ and $i_1, \ldots , i_j$ such that $(x_{i}(n))_n$ converges to $x$ for every $i \in I:=\{i_1, \ldots , i_j\}$, while $(x_{i}(n))_n \subset M \setminus B(x,\ep)$ whenever $i \in \{1 , \ldots, k+1\} \setminus I$.
    
    If $j=k+1$, that is $I = \{1 , \ldots, k+1\}$, then the set $K:=\{x_i(n) \; | \; n \in \N \text{ and } 1 \leq i \leq k+1 \} \cup \{x\} \cup \{0\}$ is a countable compact metric space such that $(\gamma_n)_n \subset \F(K)$. Thanks to \cite[Theorem~3.1]{HLP} (see also \cite{AGPP21}), $\F(K)$ has the Schur property so that $(\gamma_n)$ is actually norm null, which is what we wanted to prove. 
    
    If $j<k+1$, we let $h$ be the map defined by $h(z)=1$ if $z \in B(x, \ep/2)$ and $h(z) = 0$ if $z \in M \setminus B(x, \ep)$. It is easy to prove that $h$ is Lipschitz on $B(x,\ep/2) \cup M \setminus B(x, \ep)$ and using McShane's extension theorem (see e.g. \cite[Theorem~1.33 and Corollary 1.34]{Weaver2}), we can extend $h$ to the all $M$. Clearly, $\supp(h) \subset K:= B(x,\ep)\cup \{0\}$. Now let $T_h$ be as in Lemma~\ref{lm:multiplication_operator} and $W_h : \F(M) \to \F(K)$ be its pre-adjoint operator. It is a routine check to see that if $\mu \in \F(B(x,\ep/2) \cup \{0\} )$ then
    $W_h(\mu) = \mu$. Furthermore, there exists $N_0 \in \N$ such that for every $n \geq N_0$ and every $i \in \{i_1, \ldots , i_j\}$, $(x_{i}(n))_n \subset B(x,\ep/2)$. Thus, by construction, we have:
    $$ \forall n \geq N_0, \quad W_h\gamma_n = \underset{i \in I}{\sum_{i=1}^{k+1}} a_{i}(n) \delta(x_{i}(n)). $$
    Since $W_h$ is continuous and since $(\gamma_n)_n$ is weakly null, the sequence $(W_h\gamma_n)_n \subset \F(K)$ is weakly null as well.
    As $j<k+1$, we may use the induction hypothesis to deduce that $(W_h\gamma_n)_n$ is norm null in $\F(K)$. Recall that $\F(K)$ is a closed subspace of $\F(M)$ so that $(W_h\gamma_n)_n$ can be seen as a norm null sequence in $\F(M)$, which in turn implies that the sequence $(\mu_n)_n$ given by 
    $$ \mu_n :=  \underset{i \not\in I}{\sum_{i=1}^{k+1}} a_i(n) \delta(x_i(n)) $$
    has to be weakly null. So we use once more our induction hypothesis to get that $(\mu_n)_n$ is norm null and finally
    $$(\gamma_n)_n = (W_h\gamma_n)_n + (\mu_n)_n$$
    is norm convergent to 0 as the sum of two such sequences.

    \item There is no $i \in \{1, \ldots , k+1\}$ such that $(x_i(n))_n$ has a convergent subsequence. Then each set $\{x_i(n) \; | \; n \in \N \}$, $1 \leq i \leq k+1$, is not totally bounded. Hence there exists $\ep >0$ and an infinite subset $\mathbb{M}$ of $\N$ such that for every $i$ and every $n \neq m \in \mathbb{M}$: $d(x_i(n) , x_i(m))> \ep$. We now claim that we can extract an infinite subset $\mathbb{M}_1$ of $\mathbb{M}$ such that for every $i\neq j$ and every $n \neq m \in \mathbb{M}_1$: $d(x_i(n) , x_j(m))> \ep/2$. Let us briefly sketch this extraction. We write $\M= \{n_1 , n_2 , \ldots\}$ and we let $m_1:=n_1$. Since the sequences $(x_i(n_\ell))_\ell$, $1\leq i \leq k+1$, are $\ep$-separated, by the triangle inequality they must ``escape" the balls $B(x_j(m_1), \ep/2)$, $1\leq j \leq k+1$, eventually.  In other words, there exists $m_2 \in \M$ such that $m_1<m_2$ and for every $n \in \M$ and $1 \leq i,j \leq k+1$, $n \geq m_2 \implies d(x_i(n) , x_j(m_1)) \geq \ep/2$. By the same argument, there exists $m_3 \in \M$ such that $m_3 > m_2$ and for every $n \in \M$, $n \geq m_3 \implies d(x_i(n) , x_j(m_1)) \geq \ep/2$ and $d(x_i(n) , x_i(m_2)) \geq \ep/2$. Continuing this construction by induction provides the required $\M_1 = \{m_1, m_2 , \ldots\}$. To conclude, notice that for every $n \neq m \in \mathbb{M}_1$, $d(\supp(\gamma_n) , \supp(\gamma_m)) > \ep/2$. Since $(\gamma_n)_{n \in \M_1}$ is weakly null, we may apply Lemma~\ref{lemmaKalton} to conclude that $(\gamma_n)_{n \in \M_1}$ is norm null.
\end{itemize}

\end{proof}

\end{document}